\documentclass[a4paper]{amsart}
\usepackage[utf8]{inputenc}
\usepackage[T1]{fontenc}      
\usepackage[english]{babel} %utilisation du package français
\usepackage{graphicx}
\usepackage{mathptmx}      % use Times fonts if available on your TeX system
\usepackage{amsmath}
\usepackage{mathrsfs}% permet d'utiliser matcal
\usepackage{mathtools}% compléments à amsmath
\usepackage{amsthm} %permet de citer les thm
\usepackage{hyperref}
\usepackage{subcaption}
\usepackage{mathtools}% compléments à amsmath
\usepackage{amssymb,amsfonts}% symboles mathématiques supplémentaires
\usepackage{bm}% symbole math gras

\usepackage{soul}
\usepackage[nocompress]{cite}
\usepackage{datetime} %ajoute la date et l'heure
\usepackage{here} % permet d'utilser H, h pour place les figures
\usepackage{enumitem}

\DeclareMathOperator\EG{\underline{E}G}
\DeclareMathOperator\EGom{\underline{E}\Gom}

\DeclareMathOperator\stab{Stab}

\DeclarePairedDelimiter\gen{\langle}{\rangle}
\renewcommand*{\th}{\textsuperscript{th}}

\newcommand*{\actsGroup}{\curvearrowright}% sementic command for G acts X

\newcommand*{\bOm}{b_\omega}

\newcommand*{\cOm}{c_\omega}

\newcommand*{\dOm}{d_\omega}
\newcommand*{\iso}{\cong}% isomorphisme,commande sémantique (\simeq ou \cong)
\newcommand*{\G}{G}
\newcommand*{\Gom}{G_\omega}

\newcommand*{\h}{\mathfrak{h}}

\renewcommand*{\P}{\mathcal{P}}

\newcommand*{\stabOm}{\stab_\omega}
\newcommand*{\V}{\mathcal{V}}

\newcommand*{\X}{\mathcal{X}}

\newcommand*{\GammaMinusZero}{\widetilde{\Gamma}_+}

\newcommand*{\Tbin}{\mathcal{T}_2}

\renewcommand*{\phi}{\varphi}
\renewcommand*{\epsilon}{\varepsilon}

\newcommand*{\field}[1]{\mathbf{#1}}

\newcommand*{\Z}{\field{Z}}

\usepackage{forest}
\usetikzlibrary{angles}
\usepackage{xcolor}% couleurs
\usepackage{tikz,tikz-qtree}% tikz ! et pour les arbres
\usetikzlibrary{babel}% pour Ã©viter les clashs avec babel
\usetikzlibrary{arrows,graphs,graphs.standard,cd,calc,decorations.pathreplacing,automata}% pour les graphiques
\tikzset{mynode/.style={shape= circle, fill = white, inner sep = 0pt, outer sep = 0pt, minimum size = 4pt,draw}}
\tikzset{blacknode/.style={shape= circle, fill = white, inner sep = 0pt, outer sep = 0pt, minimum size = 4pt,draw}}
\tikzset{empty/.style={shape= circle, fill = white, inner sep = 0pt, outer sep = 0pt, minimum size = 0pt,draw}}
\tikzset{sommet/.style={shape=circle, fill=black,inner sep = 0pt, minimum size = 4pt,label distance=-24pt}}
\tikzset{sommet_blanc/.style={shape=circle, draw=black,fill=white ,inner sep = 0pt, minimum size = 4pt,label distance=-24pt}}
\tikzset{nom/.style={}}

\tikzset{myloop left/.style={loop, out=140, in = -140,min distance =6mm}}
\tikzset{myloop right/.style={loop, out=40, in = -40,min distance =6mm}}
\tikzset{myloop above/.style={loop, out=130, in = 50,min distance =6mm}}
\tikzset{myloop below/.style={loop, out=-50, in = -130,min distance =6mm}}
\tikzset{label/.style={midway, fill=white}}
\tikzset{label_above/.style={above, fill=white}}	
\tikzset{loop_above/.style={out=130,in=50,looseness=40}}
\newcommand*\Colora{red}
\newcommand*\Colorb{blue}
\newcommand*\Colorc{green}
\newcommand*\Colord{orange}
\tikzset{Ra/.style={\Colora}}% a
\tikzset{Rb/.style={\Colorb}}% b
\tikzset{Rc/.style={\Colorc}}% c
\tikzset{Rd/.style={\Colord}}% d

\theoremstyle{plain}% texte en italique
\newtheorem{lem}{Lemma}[subsection]

\newtheorem{cor}[lem]{Corollary}
\newtheorem{prop}[lem]{Proposition}

\newtheorem{question}[lem]{Question}
\newtheorem{thm}[lem]{Theorem}
\newtheorem{mainthm}{Theorem}

\theoremstyle{definition}% texte en roman

\newtheorem{defn}[lem]{Definition}

\theoremstyle{remark}% texte en roman
\newtheorem{rem}[lem]{Remark}

\usepackage{hyperref}
\usepackage{subcaption}
\title{Proper actions of Grigorchuk groups on a CAT(0) cube complex}
\author{Grégoire Schneeberger}
\begin{document}

\maketitle

\begin{abstract}
In this paper we will present a construction of a CAT(0) cube complex (an infinite cube), on which the uncountable family of Grigorchuk groups $\Gom$ act without bounded orbit. Moreover, if the sequence $\omega$ does not contain repetition, we prove that the action is proper and faithful. As a consequence of this result, this cube complex is a model for the classifying space of proper actions for all the groups $\Gom$ with $\omega$ without repetition. This construction works in a general way for any group acting on a set and which admits a commensurated subset.These examples of non-elliptic actions of infinite finitely generated torsion groups on a non-positively curved cube complex contrast to several established fixed-point theorems concerning actions of torsion groups.

\end{abstract}

\section{Introduction}

CAT(0) cube complexes are nice examples of CAT(0) spaces that share many similarities with trees. Their combinatoric and geometric properties provide useful tools to study the large class of groups which admit non-trivial actions on CAT(0) cube complexes. Groups acting on CAT(0) cube complexes were shown to play an important role in low dimensional topology \cite{Haglund2008,Haglund2012,Wise2012,Agol2013} and in geometric group theory \cite{Niblo1997,Chatterji2010,Cherix2004,Sageev2005,Chatterji2005,Nica2004,Roller1998}. 

In \cite{Grigorchuk1980}, Grigorchuk constructed an infinite finitely generated 2-group, now known as the \emph{first Grigorchuk group}, and showed that this group has a lot of exotic properties. Most notably, it is of intermediate growth and hence amenable, but not elementary amenable. This example was generalized in \cite{Grigorchuk1984} to an uncountable family of groups $\{\Gom\}_{\omega \in \{0,1,2\}^\infty}$, the \emph{Grigorchuk groups}, that have since generated a lot of research. Grigorchuk groups are defined by their action by automorphisms on the infinite binary rooted tree. This action extends naturally to an action by homeomorphisms on the boundary of the tree. Schreier graphs of stabilizers of points in this action have linear structure and have proved an important tool in the study of these groups, see e.g. \cite{Grigorchuk2019,Bartholdi2000,Grigorchuk2018,Vorobets2012}

In \cite{Sageev1995}, Sageev showed that if a finitely generated group has a Schreier graph with more than one end, then one can construct a CAT(0) cube complex on which the group acts without bounded orbit. The way in which the complex and the action are constructed is explicitly described when an almost invariant subset is given, which is the case if a description of a multi-ended Schreier graph is available.

It is well known that Grigorchuk groups have Schreier graphs with 2 ends and hence it is tempting to apply Sageev's construction and to investigate the complex that it gives and the action of $\Gom$ on it for different sequences $\omega$. Recall that the groups $\Gom$ are not quasi-isometric and may have different algebraic and geometric properties, for example, being torsion groups or not. We will show explicitly how all groups $\Gom$ act on an infinite cube $\X$ (isometric to all the $\X_\omega$) without bounded orbits\footnote{It is always possible to find a common cubulation for groups, since every CAT(0) cube complex can be equivariantly embedded into an infinite cube. In other words, all the (countable) cubulable groups act on the (countably) infinite cube.}, see Theorem~\ref{T:bounded_orbit}.

Moreover, we can strengthen this result for an uncountable subfamily of groups $\Gom$, including the self-similar examples $\G_{(012)^\infty}$ (the first Grigorchuk group which is an infinite 2-group) and $\G_{(01)^\infty}$ (which has an element of infinite order), by proving that the actions of these groups on $\X$ are faithful and proper\footnote{Here, an action is called proper if the cubes' stabilizers are all finite.}, see Theorem~\ref{T:Properness_action}. 

All Grigorchuk groups $\Gom$ are amenable and therefore have the Haagerup property \cite{Bekka1995}. It is known that groups with the Haagerup property satisfy the Baum-Connes conjecture \cite{Higson1997}. Recall that this conjecture links the $K$-homology of the classifying space of proper actions of a group $G$ and the $K$-theory of its reduced $C^*$-algebra by conjecturing that the assembly map $\mu_i : RK_i^G( \underline{E}G) \rightarrow K_i (C^*_r(G))$ is an isomorphism for $i = 0,1$, see \cite{Valette2002}. Even if the conjecture is proved for a particular group, it is interesting to be able to describe the isomorphisms $\mu_i$. For this, we need a model for the classifying space of proper actions $\underline{E}G$. Theorem~\ref{T:model} shows that the complex $\X$ gives such a model for all the groups $\Gom$ whose action on $\X$ is proper.

From the more global point of view of the geometric group theory, there are various results which show that purely elliptic actions of finitely generated groups on non-positively curved complexes are automatically elliptic, see \cite{Haettel2022} and the references therein. However, the construction presented in this paper diverges from these results by providing examples of purely periodic actions without fixed points of finitely generated groups.

In reality, this construction applies more generally to groups acting transitively on a set and admitting a commensurated subset, see Theorem\ref{T:main_A}.We shall present this general construction in Section~\ref{S:General}, before applying it to the case of Grigorchuk groups in the reminding sections, see Theorem~\ref{T:bounded_orbit} and Theorem~\ref{T:Properness_action}.

%%%%%%%%%%%%%%%%%%%%%%%%%%%%%%%%%%%%%%%%%%%%%%%%%%%%%%%%%%%%%%%%%%%%%%

\section{General case}~\label{S:General}
\subsection{CAT(0) cube complexes}

Let us recall some useful facts and fixing notations about cube complexes and group actions. A cube complex is a space obtained by gluing, via isometries, euclidean cubes with edges of length one. This space can be equipped with a metric induced by that of the cubes. If the complex is simply connected and if the link of each cube does not contain the boundary of a triangle which is not filled then the metric verifies the CAT(0) condition, see \cite[Theorem B.8]{Leary2013}. 

A key feature of CAT(0) cube complexes is the existence of hyperplanes. Two edges are said equivalent if they are the opposite sides of a square. An equivalent class of edges is a \emph{hyperplane}. It is also possible to define hyperplanes in a more geometrical way, as the codimension-1 subspace spanned by the midpoints of the edges. In the case of a CAT(0) cube complex, hyperplanes are also CAT(0) cube complexes which split the complex into exactly two connected components called \emph{half-spaces}, see \cite{Sageev1995} for details.

The following proposition will be useful. It is a generalization, without any restriction on dimension or local finiteness, of the fixed point theorems for actions with bounded orbits. 

\begin{prop}\label{P:bounded_orbits_fix_vertex}
Let $G$ be a group acting on a CAT(0) cube complex $X$. If the action has a bounded orbit and no hyperplane-inversion (i.e. no element sends a half space to its complement), then $G$ fixes a vertex of $X$.
\end{prop}

\begin{proof}
It is known that if a group acts with bounded orbits on a CAT(0) cube complex, there exists a cube $Q$ which is stabilized by the action. If the dimension of $Q$ is equal to $0$, we have our fixed point. If the dimension is greater than $0$, let $\h$ be a hyperplane which crosses $Q$ and $\h_+$ be one of the two halfspaces delimited by $\h$. As the action does not contain any inversion, the set of halfspaces $\{g\h_+ : g \in  G \}$ pairwise intersect. Using the Helly property, the set 
\begin{equation*}
Q \cap \bigcap_{g \in G} g\h_+
\end{equation*}
is not empty and provides a proper subcube of $Q$ which is stabilized by the action of $G$. We can find a fixed vertex by iterating the process. 
\end{proof}

\subsection{Commensurated subset and action on a cube complex}

Following Cornulier's work \cite{Cornulier2013}, we will use the fact that the main ingredient of this construction is the existence of a group action on a set that has a commensurated subset.

\begin{defn}
Let $G$ be a group which acts on a set $X$. A subset $Y$ of $X$ is \emph{commensurated} by the action of $G$ if 
\begin{equation*}
\left| gY \triangle Y \right| < \infty
\end{equation*}
for all $g$ in $G$.
\end{defn}

In this case, we will prove the following result

\begin{mainthm}\label{T:main_A}
Let $G$ be a group acting on a set $S$ and $R \subset S$ a subset. Assume that:
\begin{itemize}
\item G acts transitively on $S$.
\item $R$ is commensurated by $G$ and $R$ and $S \setminus R$ are both infinite.
\end{itemize}

Then $G$ acts on an infinite cube $X$ without bounded orbit and every subgroup $H < G$ with bounded orbits fixes a vertex.

Moreover, if
\begin{itemize}
\item $\stab(T)$ is finite for every $T \subset S$ such that $T \triangle R$ is finite,
\end{itemize}
then $G$ acts properly on $X$.
\end{mainthm}
\begin{proof}
We consider $\P(S)$ the set of subsets of $S$ which can be seen as 
\begin{equation*}
\P(S) \iso \{0,1\}^S \iso \left\{ \phi : S \rightarrow \{0,1\} \right\}.
\end{equation*}
This set can naturally be endowed with a cube complex structure. Each subset of $S$ is a vertex and two vertices are connected by an edge if and only if their symmetric difference contains one element. We use this element to label the edge. The hyperplanes consists of the sets of edges with the same label. For every element $s$ of $S$, the associated hyperplane $\h_s$ cuts the complex into two halfspaces $\h^0_s$ and $\h^1_s$. The vertices contained in these halfspaces can be explicitly described 
\begin{equation*}
\h^0_s = \left\{ \phi \in \P(S) : \phi(s) = 0 \right \} \text{ and } \h^1_s = \left\{ \phi \in \P(S) : \phi(s) = 1 \right \}.
\end{equation*}

If $S$ is infinite, the cube complex has infinite dimension and is neither locally finite nor path connected.  

The action $G \actsGroup S$ can be extended to an action $G \actsGroup \P(S)$ via
\begin{equation*}
g\phi(s) = \phi(g^{-1}s).
\end{equation*}

Let us consider $X$ the connected component which contains the vertex associated to the subset $R$. It should be observed that since $R$ and $S \setminus R$ are infinite, the preimage of $0$ and $1$ are infinite for all the vertices of $X$ (viewed as functions). 

As $R$ is commensurated by the action of $G$, the action $G \actsGroup \P(S)$ can be restricted to an action $G \actsGroup X$. Combining the descriptions of the halfspaces, the facts that $G$ acts transitively on $X$ and that the functions associated to the vertices of $X$ are equal to 0 and 1 one infinitely many elements of $S$, we observe that the action does not contain a hyperplane inversion. Then, by the Proposition~\ref{P:bounded_orbits_fix_vertex}, if a subgroup $H$ has a bounded orbit, it fixes a vertex. However, the only vertices which are fixed by the group $G$ are those associated with $\emptyset$ and $S$, but they are not in $X$. Then, the group $G$ acts without bounded orbit.
The second part of the proposition is clear. 
\end{proof}

We will now apply Theorem~\ref{T:main_A} to Grigorchuk groups. To do this, we first need to construct an action on a set and find a set that is commensurated. Then we will show that the action is proper in some cases.

\section{Grigorchuk groups and group action}
\subsection{Grigorchuk groups}
Let $\Tbin$ denote the infinite binary rooted tree. The vertices of this tree can be described as (finite) binary sequences, see Figure~\ref{F:Tbin}. 
\begin{figure}
\centering
\begin{tikzpicture}[scale=0.7]
\node[nom] (E) at (0,0) {$\emptyset$};
\node[nom] (0) at (-2,-2) {$0$};
\node[nom] (1) at (2,-2) {$1$};
\node[nom] (00) at (-3,-4) {$00$};
\node[nom] (01) at (-1,-4) {$01$};
\node[nom] (10) at (1,-4) {$10$};
\node[nom] (11) at (3,-4) {$11$};
\node[nom] (000) at (-3.5,-6) {$000$};
\node[nom] (001) at (-2.5,-6) {$001$};
\node[nom] (010) at (-1.5,-6) {$010$};
\node[nom] (011) at (-0.5,-6) {$011$};
\node[nom] (100) at (0.5,-6) {$100$};
\node[nom] (101) at (1.5,-6) {$101$};
\node[nom] (110) at (2.5,-6) {$110$};
\node[nom] (111) at (3.5,-6) {$111$};
\draw (E) to (0) ;
\draw (E) to (1) ;
\draw (0) to (00) ;
\draw (0) to (01) ;
\draw (1) to (10) ;
\draw (1) to (11) ;
\draw (00) to (000) ;
\draw (00) to (001) ;
\draw (01) to (010) ;
\draw (01) to (011) ;
\draw (10) to (100) ;
\draw (10) to (101) ;
\draw (11) to (110) ;
\draw (11) to (111) ;
\end{tikzpicture}
\caption{The first levels of $\Tbin$}\label{F:Tbin}.
\end{figure}
\begin{defn}
Let $\omega = \omega_1\omega_2\ldots$ be an infinite sequence in $\{0,1,2\}^\infty$. Denote by $\sigma$ the shift on the space $\{0,1,2\}^\infty$. The group $\Gom$ is the subgroup of the automorphisms group of $\Tbin$ generated by the automorphisms $a,\bOm,\cOm,\dOm$ where,
\begin{equation*}
a(0x) = 1x \quad \text{and} \quad a(1x) = 0x
\end{equation*}
for all binary sequences $x$ and
\begin{align*}
\bOm(0x) &= \begin{cases} 0a(x)	& \text{if } \omega_1 \neq 2 \\ 0x		& \text{if } \omega_1 = 2 \end{cases} & \bOm(1x) = 1 b_{\sigma \omega}(x) \\
\cOm(0x) &= \begin{cases} 0a(x)	& \text{if } \omega_1 \neq 1 \\ 0x		& \text{if } \omega_1 = 1 \end{cases} & \cOm(1x) = 1 c_{\sigma \omega}(x) \\
\dOm(0x) &= \begin{cases} 0a(x)	& \text{if } \omega_1 \neq 0 \\ 0x		& \text{if } \omega_1 = 0 \end{cases} & \dOm(1x) = 1 d_{\sigma \omega}(x).
\end{align*}
\end{defn}
The actions of these generators can be represented graphically, see Figure~\ref{F:def_grig_group}.
\begin{figure}
\centering
\begin{tikzpicture}[scale=0.7]
\begin{scope}[xscale=0.7, xshift=7cm, yshift=4cm]
\node[nom] at (0,0.5) {$\emptyset$};
\node[nom] at (-3,-1) {$a$};
\node[sommet] (E) at (0,0) {};
\node[sommet] (0) at (-2,-2) {};
\node[sommet] (1) at (2,-2) {};

\draw (E) to (0) ;
\draw (E) to (1) ;
\draw[bend right, dashed, <->] (0) to (1);
\end{scope}
\begin{scope}[xscale=0.7,xshift=0cm]
\node[nom] at (0,0.5) {$\emptyset$};
\node[sommet] (E) at (0,0) {};
\node[sommet] (0) at (-2,-2) {};
\node[sommet] (1) at (2,-2) {};
\node[sommet] (00) at (-3,-4) {};
\node[sommet] (01) at (-1,-4) {};
\node[sommet] (10) at (1,-4) {};
\node[sommet] (11) at (3,-4) {};
\node[nom] at (2,-3.5) {$b_{\sigma \omega}$};
\node[nom] at (-2,-3.5) {$I^2_\omega$};
\node[nom] at (-3,-1) {$b_{\omega}$};
\draw (E) to (0) ;
\draw (E) to (1) ;
\draw (0) to (00) ;
\draw (0) to (01) ;
\draw (1) to (10) ;
\draw (1) to (11) ;
\end{scope}
\begin{scope}[xscale=0.7,xshift=7cm]
\node[nom] at (0,0.5) {$\emptyset$};
\node[sommet] (E) at (0,0) {};
\node[sommet] (0) at (-2,-2) {};
\node[sommet] (1) at (2,-2) {};
\node[sommet] (00) at (-3,-4) {};
\node[sommet] (01) at (-1,-4) {};
\node[sommet] (10) at (1,-4) {};
\node[sommet] (11) at (3,-4) {};
\node[nom] at (2,-3.5) {$c_{\sigma \omega}$};
\node[nom] at (-2,-3.5) {$I^1_\omega$};
\node[nom] at (-3,-1) {$c_{\omega}$};
\draw (E) to (0) ;
\draw (E) to (1) ;
\draw (0) to (00) ;
\draw (0) to (01) ;
\draw (1) to (10) ;
\draw (1) to (11) ;
\end{scope}
\begin{scope}[xscale=0.7,xshift=14cm]
\node[nom] at (0,0.5) {$\emptyset$};
\node[sommet] (E) at (0,0) {};
\node[sommet] (0) at (-2,-2) {};
\node[sommet] (1) at (2,-2) {};
\node[sommet] (00) at (-3,-4) {};
\node[sommet] (01) at (-1,-4) {};
\node[sommet] (10) at (1,-4) {};
\node[sommet] (11) at (3,-4) {};
\node[nom] at (2,-3.5) {$d_{\sigma \omega}$};
\node[nom] at (-2,-3.5) {$I^0_\omega$};
\node[nom] at (-3,-1) {$d_{\omega}$};
\draw (E) to (0) ;
\draw (E) to (1) ;
\draw (0) to (00) ;
\draw (0) to (01) ;
\draw (1) to (10) ;
\draw (1) to (11) ;
\end{scope}
\end{tikzpicture}
\caption{The actions of the generators, where
$
I_\omega^i = 
\begin{cases}
1	& \omega_1 = i \\
a	& \omega_1 \neq i
\end{cases}.
$}\label{F:def_grig_group}
\end{figure}

Let $v$ be a vertex of the $n$\th{} level of the tree and $g \in \Gom$ an element which fix $v$. We can look at the action of $g$ on the subtree rooted at $v$. This subtree is isometric to $\Tbin$ and then, we can define a new automorphism, which is an element of $\G_{\sigma^n \omega}$. This automorphism is called the \emph{restriction of $g$ on $v$} and is denoted by $g_v$. 

An element $g$ \emph{stabilizes the $n$\th{} level} if $g(v) = v$ for every $v$ on the $n$\th{} level. The subgroup of all such elements in $\Gom$ is $\stabOm(n)$, the \emph{stabilizer of the $n$\th{} level}. For an element $g$ in $\stabOm(1)$, we denote by $(g_0,g_1)$ the two restrictions of $g$ on the two subtrees rooted at the first level. If $g$ does not stabilize the first level, we can still write it as $a(g_0,g_1)$. We refer the reader to the section VII of \cite{DelaHarpe2000} and \cite{Grigorchuk2006} for a detailed general exposition of these concepts.

\subsection{Construction of the action}

The action of $\Gom$ can be extended to $\partial \Tbin$, the boundary of the tree. The points of the boundary $\partial \Tbin$ of the tree can be seen as one-sided binary sequences. Consider the orbit of $0^\infty$ and define $\Gamma_\omega$ as the graph with $\Gom 0^\infty$ as vertices and $\{(g0^\infty, sg0^\infty) : g \in \Gom, s \in \{a, \bOm, \cOm, \dOm \}\}$ as edges. This is the orbital graph of the action of $\Gom$ on $0^\infty$. These graphs are not isomorphic if we look at them as labelled\footnote{By labelled graph we mean a graph with edges labelled by generators $\{a,\bOm,\cOm,\dOm\}$} graphs, but the underlying unlabelled graphs are the same.
\begin{prop}
The vertices of the graphs $\Gamma_\omega$ are exactly the infinite sequences with a finite number of $1$. Moreover all the $\Gamma_\omega$ are equal as unlabelled graphs.
\end{prop}
\begin{proof}
These facts are proved in \cite{MatteBon2015} for the orbital graph of $1^\infty$, but the proofs can be adapted directly.
\end{proof}
\begin{defn}
We define the graph $\Gamma$ as the underlying unlabelled graph of the graphs $\Gamma_\omega$ and $V(\Gamma)$ as its set of vertices.
\end{defn}

See Figure~\ref{F:Graphe_gamma} for some examples of graphs $\Gamma_\omega$ and $\Gamma$.
 
\begin{figure}
\centering
\begin{tikzpicture}[scale=0.6]
\begin{scope}[yshift=6cm]
\node[nom] (Name) at (-10,0) {$\Gamma_{(012)^\infty}$};
\node[sommet] (1101) at (-6,0) {};
\node[nom]  at (-6,-0.75) {$11010^\infty$};
\node[sommet] (11) at (-4,0) {};
\node[nom]  at (-4,-0.75) {$110^\infty$};
\node[sommet] (01) at (-2,0) {};
\node[nom]  at (-2,-0.75) {$010^\infty$};
\node[sommet] (0) at (0,0) {};
\node[nom]  at (0,-0.75) {$0^\infty$};
\node[sommet] (10) at (2,0) {};
\node[nom]  at (2,-0.75) {$10^\infty$};
\node[sommet] (101) at (4,0) {};
\node[nom]  at (4,-0.75) {$1010^\infty$};
\node[style=empty] (L) at (-7,0) {};
\node[style=empty](R) at (5,0) {};
\draw[\Colora] (101) to  (R) ;
\draw[bend right,\Colord] (10) to   (101);
\draw[bend left,\Colorb] (10) to  (101);
\draw[loop_above,\Colorc] (10) to   (10);
\draw[loop_above,\Colorc] (101) to (101);
\draw[\Colora] (0) to  (10);
\draw[loop_above,\Colord] (0) to   (0);
\draw[bend right,\Colorb] (0) to    (01) ;
\draw[bend left,\Colorc] (0) to   (01) ;
\draw[loop_above,\Colord] (01) to  (01) ;
\draw[\Colora] (01) to  (11);
\draw[loop_above,\Colorb] (11) to  (11) ;
\draw[bend right,\Colord] (11) to   (1101) ;
\draw[bend left,\Colorc] (11) to  (1101) ;
\draw[loop_above,\Colorb] (1101) to  (1101) ;
\draw[\Colora] (1101) to  (L);
\end{scope}
\begin{scope}[yshift=3cm]
\node[nom] (Name) at (-10,0) {$\Gamma_{(01)^\infty}$};
\node[sommet] (1101) at (-6,0) {};
\node[nom]  at (-6,-0.75) {$11010^\infty$};
\node[sommet] (11) at (-4,0) {};
\node[nom]  at (-4,-0.75) {$110^\infty$};
\node[sommet] (01) at (-2,0) {};
\node[nom]  at (-2,-0.75) {$010^\infty$};
\node[sommet] (0) at (0,0) {};
\node[nom]  at (0,-0.75) {$0^\infty$};
\node[sommet] (10) at (2,0) {};
\node[nom]  at (2,-0.75) {$10^\infty$};
\node[sommet] (101) at (4,0) {};
\node[nom]  at (4,-0.75) {$1010^\infty$};
\node[style=empty] (L) at (-7,0) {};
\node[style=empty](R) at (5,0) {};
\draw[\Colora] (101) to  (R) ;
\draw[bend right,\Colord] (10) to   (101);
\draw[bend left,\Colorb] (10) to  (101);
\draw[loop_above,\Colorc] (10) to   (10);
\draw[loop_above,\Colorc] (101) to (101);
\draw[\Colora] (0) to  (10);
\draw[loop_above,\Colord] (0) to   (0);
\draw[bend right,\Colorb] (0) to    (01) ;
\draw[bend left,\Colorc] (0) to   (01) ;
\draw[loop_above,\Colord] (01) to  (01) ;
\draw[\Colora] (01) to  (11);
\draw[loop_above,\Colord] (11) to  (11) ;
\draw[bend right,\Colorb] (11) to   (1101) ;
\draw[bend left,\Colorc] (11) to  (1101) ;
\draw[loop_above,\Colord] (1101) to  (1101) ;
\draw[\Colora] (1101) to  (L);
\end{scope}
\begin{scope}
\node[nom] (Name) at (-10,0) {$\Gamma$};
\node[sommet] (1101) at (-6,0) {};
\node[nom]  at (-6,-0.75) {$11010^\infty$};
\node[sommet] (11) at (-4,0) {};
\node[nom]  at (-4,-0.75) {$110^\infty$};
\node[sommet] (01) at (-2,0) {};
\node[nom]  at (-2,-0.75) {$010^\infty$};
\node[sommet] (0) at (0,0) {};
\node[nom]  at (0,-0.75) {$0^\infty$};
\node[sommet] (10) at (2,0) {};
\node[nom]  at (2,-0.75) {$10^\infty$};
\node[sommet] (101) at (4,0) {};
\node[nom]  at (4,-0.75) {$1010^\infty$};
\node[style=empty] (L) at (-7,0) {};
\node[style=empty] (R) at (5,0) {};
\draw (101) to (R) ;
\draw (10) to [bend right]  (101);
\draw (10) to [bend left] (101);
\draw (10) to [loop_above]  (10);
\draw (101) to [loop_above]  (101);
\draw (0) to  (10);
\draw (0) to [loop_above]  (0);
\draw (0) to [bend right]   (01) ;
\draw (0) to [bend left]  (01) ;
\draw (01) to [loop_above] (01) ;
\draw (01) to  (11);
\draw[loop_above] (11) to  (11) ;
\draw[bend right] (11) to   (1101) ;
\draw[bend left] (11) to  (1101) ;
\draw[loop_above] (1101) to  (1101) ;
\draw (1101) to  (L);
\end{scope}
\end{tikzpicture}
\caption{Some examples of graphs $\Gamma_\omega$ and the graph $\Gamma$. The red edges are labelled by $a$, the blue by $\bOm$, the green by $\cOm$ and the orange by $\dOm$.}
\label{F:Graphe_gamma}
\end{figure}

We will now define the commensurated subset. Let $\Gamma_+$ be the set of vertices of $\Gamma$ including $0^\infty$ and all the vertices on the right side of it, see Figure~\ref{F:gamma_+}. 

\begin{figure}
\centering
\begin{tikzpicture}
\node[sommet_blanc] (1101) at (-6,0) {};
\node[nom]  at (-6,-0.75) {$11010^\infty$};
\node[sommet_blanc] (11) at (-4,0) {};
\node[nom]  at (-4,-0.75) {$110^\infty$};
\node[sommet_blanc] (01) at (-2,0) {};
\node[nom]  at (-2,-0.75) {$010^\infty$};
\node[sommet] (0) at (0,0) {};
\node[nom]  at (0,-0.75) {$0^\infty$};
\node[sommet] (10) at (2,0) {};
\node[nom]  at (2,-0.75) {$10^\infty$};
\node[sommet] (101) at (4,0) {};
\node[nom]  at (4,-0.75) {$1010^\infty$};
\node[style=empty] (L) at (-7,0) {};
\node[style=empty] (R) at (5,0) {};
\draw (101) to (R) ;
\draw (10) to [bend right]  (101);
\draw (10) to [bend left] (101);
\draw (10) to [loop_above]  (10);
\draw (101) to [loop_above]  (101);
\draw (0) to  (10);
\draw (0) to [loop_above]  (0);
\draw (0) to [bend right]   (01) ;
\draw (0) to [bend left]  (01) ;
\draw (01) to [loop_above] (01) ;
\draw (01) to  (11);
\draw[loop_above] (11) to  (11) ;
\draw[bend right] (11) to   (1101) ;
\draw[bend left] (11) to  (1101) ;
\draw[loop_above] (1101) to  (1101) ;
\draw (1101) to  (L);
\end{tikzpicture}
\caption{The black vertices are elements of $\Gamma_+$} \label{F:gamma_+}
\end{figure}
The vertices of $\Gamma_+$ can be described in an explicit way as follows,
\begin{prop}\label{P:Gamma_description}
The set $\Gamma_+$ consists of $0^\infty$ and all the vertices of the form $x = x_1 x_2 \ldots x_n 0^\infty$ with $x_n = 1$ and $n$ odd.
\end{prop}
\begin{proof}
Let us begin by an observation. Given any element $x$ of $\{0,1\}^\infty$, a generator $s$ of $\Gom$ can only acts non-trivially in two ways: 
\begin{itemize}%[label=\textbf{(M\arabic*})]
\item flip the first digit of $x$ (if $s = a$) %\label{M1}
\item flip the digit after the first apparition of $0$ (if $s \in \{\bOm,\cOm,\dOm\}$) %\label{M2}
\end{itemize}
It is straightforward to see that these two moves do not change the parity of the position of the last $1$, except for $0^\infty$ and $10^\infty$. We constructed $\Gamma$ in a way to have $10^\infty$ on the right of $0^\infty$, and then, the parity of the position of the last $1$ is the same as for $10^\infty$ for all the other vertices on the right of $0^\infty$.
\end{proof}

\begin{prop}\label{P:Stab_connected_comp}
The set $\Gamma_+$ is commensurated by the action of $\Gom$.
\end{prop}
\begin{proof}
We need to prove that $g\Gamma_+ \triangle \Gamma_+$ is finite for every $g$ in $\Gom$. 

Let $Y=\Gamma_+ \setminus g\Gamma_+$. If $n$ is the length of $g$ with respect to the standard generating set, a vertex $x$ in $Y$ must be at distance at most $n$ of $\Gamma_+^c$ in the graph $\Gamma$. As this graph is locally finite, $Y$ is necessarily finite. Similarly $g\Gamma_+ \setminus \Gamma_+$ is finite and so $g\Gamma_+ \triangle \Gamma_+$ is also finite.
\end{proof}
The following theorem is now a direct consequence of the first part of the previous proposition and Theorem~\ref{T:main_A}.
\begin{mainthm}\label{T:bounded_orbit}
There exist an infinite CAT(0) cube complex $\X$ on which every $\Gom$ acts without bounded orbit, for every $\omega$ in $\{0,1,2\}^\infty$.
\end{mainthm}

In the following, we will keep the notation $\X$ for this cube complex which is the connected component of $\Gamma_+$ in the infinite cube constructed from the set $\Gamma$ (see section~\ref{S:General} for details).
\begin{rem}\label{R:dimension}
We can ask if it is possible to construct an action of $\Gom$ without bounded orbit on a CAT(0) cube complex of finite dimension, or at least of locally finite dimension. A negative answer can be given for the groups $\Gom$ whose sequences $\omega$ contain an infinity of $0,1$ and $2$. Indeed, such groups are 2-groups, see \cite{Grigorchuk1984}, and the following theorem shows that they cannot act on smaller CAT(0) cube complexes without bounded orbit.%
\end{rem}
\begin{thm}[{\cite{Genevois2024}}]%\label{T:Periodic_actions}
Let $G$ be a group acting on a CAT(0) cube complex $X$. Assume that there is a finite number of orbits of hyperplanes and that $X$ does not contain an Hilbert cube\footnote{A Hilbert cube is the product of countably infinitely many copies of intervals $[0,1]^n$}. If the action is purely periodic, i.e. every element $g$ defines a periodic isometry of $X$, then $G$ stabilizes a finite dimensional cube.
\end{thm}

\subsection{Properness and Faithfulness}

In this section, we will consider the uncountable subfamily of Grigorchuk groups $\Gom$ which are indexed by sequences $\omega$ in $\{0,1,2\}^\infty$ without repetition, i.e. $\omega_i \neq \omega_{i+1}$ for every $i$, and prove that their action on $\X$ is proper and faithful by using the second part of Theorem~\ref{T:main_A}.

The idea of the proof is to show that stabilizing a subset of vertices of $\Gamma$ is a strong condition and only few elements of $\Gom$ succeed. We will decompose this proof into different steps. Firstly, we will compute the stabilizers of two particular subsets of $V(\Gamma)$. Secondly, we will show that the stabilizer of a subset of vertices of $\Gamma$ which has a finite symmetric difference with the two particular subsets, is also finite. Finally, we will explain how these stabilizers of subsets can be related with the stabilizers of the cubes of $\X$.

Let us begin by defining precisely the notion of stabilizer of a subset of $V(\Gamma)$ and the two particular subsets that we will study.
\begin{defn}
Let $\Omega \subset V(\Gamma)$ be a subset of vertices of $\Gamma$ and $\Gom$ a Grigorchuk group. The stabilizer of $\Omega$ for the action of $\Gom \actsGroup \X$ is
\begin{equation*}
\stabOm(\Omega) = \{g \in \Gom : g\Omega = \Omega \}.
\end{equation*} 
\end{defn}
\begin{defn}
We will study two particular subsets of $V(\Gamma)$ : 
\begin{itemize}
\item $\Gamma_+$, which we have already described above
\item $\GammaMinusZero = \Gamma_+ \setminus \{0^\infty \}$.
\end{itemize}
\end{defn}
The following lemma explains the behaviour of the vertices of $\Gamma$ between these two subsets when we add a digit at the beginning of their binary writing. 
\begin{lem} \label{L:prefix_gamma}
Let $x$ be a vertex of $\Gamma$. Then
\begin{itemize}
\item $x \in \Gamma_+ \Leftrightarrow 0 x \in \GammaMinusZero^c$.
\item $x \in \Gamma_+^c \Leftrightarrow 0x \in \GammaMinusZero$.
\item $x \in \Gamma_+^c \Rightarrow 1x \in \GammaMinusZero$.
\item $x \in \GammaMinusZero \Leftrightarrow 0x \in \Gamma_+^c$.
\item $x \in \GammaMinusZero \Leftrightarrow 1x \in \GammaMinusZero^c$.
\item $x \in \GammaMinusZero^c \Leftrightarrow 0x \in \Gamma_+$.
\item $x \in \GammaMinusZero^c \Leftrightarrow 1x \in \GammaMinusZero$.
\end{itemize}
\end{lem}
\begin{proof}
We proved in Proposition~\ref{P:Gamma_description} that $\Gamma_+$ can be fully described by looking at the parity of the position of the last digit 1. This allows us to give the following descriptions: 
\begin{align*}
\Gamma_+ &= \{0^\infty\} \cup \{x_1x_2 \ldots x_{2n+1}0^\infty\} \\
\Gamma_+^c &= \{x_1x_2 \ldots x_{2n}0^\infty\} \\
\GammaMinusZero &=  \{x_1x_2 \ldots x_{2n+1}0^\infty \} \\ 
\GammaMinusZero^c &=  \{0^\infty\} \cup\{x_1x_2 \ldots x_{2n}0^\infty \}
\end{align*}
where every last $x_i$ is equal to $1$. Every point of the lemma can be proved by direct computations using these descriptions.

\end{proof}
We will study the restrictions of elements of $\stabOm(\Gamma_+)$ and $\stabOm(\GammaMinusZero)$ on the first level of the tree.
\begin{lem}\label{L:stab_projections}
Let $\omega$ be a sequence in $\{0,1,2\}^\infty$ and $g$ an element of $\Gom$.
\begin{enumerate}
\item If $g$ stabilizes the first level of the tree and $g \in \stabOm(\Gamma_+)$, then $g_0,g_1 \in \stab_{\sigma\omega}(\GammaMinusZero)$.
\item If $g$ stabilizes the first level of the tree and $g \in \stabOm(\GammaMinusZero)$, then $g_0 \in \stab_{\sigma\omega}(\Gamma_+)$ and $g_1 \in \stab_{\sigma\omega}(\Gamma_+) \cap \stab_{\sigma\omega}(\GammaMinusZero)$.
\item If $g$ does not stabilize the first level of the tree and $g \in \stabOm(\GammaMinusZero)$, then $g_0,g_1 \in \stab_{\sigma\omega}(\Gamma_+) \cap \stab_{\sigma\omega}(\GammaMinusZero)$.
\end{enumerate}

\end{lem}
\begin{proof}
Let's start by noting that an element $g$ is in $\stabOm(\Gamma_+)$ (resp. $\stabOm(\GammaMinusZero)$) if and only if it is in $\stabOm(\Gamma_+^c)$ (resp. $\stabOm(\GammaMinusZero^c)$). The proof is an application of Lemma~\ref{L:prefix_gamma}. 
\begin{enumerate}
\item 
Let $g$ be an element of $\stabOm(\Gamma_+)$ which stabilizes the first level of the tree. Then,
\begin{align*}
x \in \GammaMinusZero 	& \Rightarrow  0 x \in \Gamma_+^c \\
						& \Rightarrow  g(0 x) \in \Gamma_+^c \\
						& \Rightarrow  0g_0(x) \in \Gamma_+^c \\
						& \Rightarrow  g_0(x) \in \GammaMinusZero \\
						& \Rightarrow  g_0 \in \stab_{\sigma\omega}(\GammaMinusZero),
\end{align*}
\begin{align*}
x \in \GammaMinusZero^c 	& \Rightarrow  1 x \in \Gamma_+ \\
							& \Rightarrow  g(1 x) \in \Gamma_+ \\
							& \Rightarrow  1g_1(x) \in \Gamma_+ \\
							& \Rightarrow  g_1(x) \in \GammaMinusZero^c  \\
							& \Rightarrow  g_1 \in \stab_{\sigma\omega}(\GammaMinusZero),
\end{align*}
\item 
Let $g$ be an element of $\stabOm(\GammaMinusZero)$ which stabilizes the first level of the tree. Then,
\begin{align*}
x \in \Gamma_+^c 	&\Rightarrow  0 x \in \GammaMinusZero \\
					&\Rightarrow  g(0 x) \in \GammaMinusZero \\
					&\Rightarrow  0g_0(x) \in \GammaMinusZero \\
					&\Rightarrow  g_0(x) \in \Gamma_+^c \\
					&\Rightarrow  g_0 \in \stab_{\sigma\omega}(\Gamma_+),
\end{align*}
\begin{align*}
x \in \GammaMinusZero & \Rightarrow  1 x \in \GammaMinusZero^c \\
						&\Rightarrow  g(1 x) \in \GammaMinusZero^c  \\
						& \Rightarrow  1g_1(x) \in \GammaMinusZero^c \\ 
						& \Rightarrow  g_1(x) \in \GammaMinusZero \\
						& \Rightarrow  g_1 \in \stab_{\sigma\omega}(\GammaMinusZero), 
\end{align*}
\begin{align*}
x \in \Gamma_+^c	&\Rightarrow 1x \in \GammaMinusZero \\
					&\Rightarrow g(1x) \in \GammaMinusZero \\
					&\Rightarrow 1g_1(x) \in \GammaMinusZero \\
					&\Rightarrow g_1(x) \in \Gamma_+^c \\
					&\Rightarrow g_1 \in \stab_{\sigma\omega}(\Gamma_+).
\end{align*}
\item 
Let $g$ be an element of $\stabOm(\GammaMinusZero)$ which does not stabilize the first level of the tree. Then,
\begin{align*}
x \in \Gamma_+			&\Rightarrow	0x \in \GammaMinusZero^c \\
							&\Rightarrow	g(0x) \in \GammaMinusZero^c \\
							&\Rightarrow	1g_0(x) \in \GammaMinusZero^c \\
							&\Rightarrow 	g_0(x) \in \GammaMinusZero \subset \Gamma_+ \\
							&\Rightarrow  	g_0 \in \stab_{\sigma\omega}(\Gamma_+),
\end{align*}
\begin{align*}
x \in \GammaMinusZero\subset \Gamma_+		&\Rightarrow	0x \in \GammaMinusZero^c \\
							&\Rightarrow	g(0x) \in \GammaMinusZero^c \\
							&\Rightarrow	1g_0(x) \in \GammaMinusZero^c \\
							&\Rightarrow	g_0(x) \in \GammaMinusZero \\
							&\Rightarrow  	g_0 \in \stab_{\sigma\omega}(\GammaMinusZero),
\end{align*}
\begin{align*}
x \in \Gamma_+^c			&\Rightarrow	1x \in \GammaMinusZero \\
							&\Rightarrow	g(1x) \in \GammaMinusZero \\
							&\Rightarrow	0g_1(x) \in \GammaMinusZero \\
							&\Rightarrow 	g_1(x) \in \Gamma_+^c \\
							&\Rightarrow  	g_1 \in \stab_{\sigma\omega}(\Gamma_+),
\end{align*}
\begin{align*}
x \in \GammaMinusZero^c 	&\Rightarrow	1x \in \GammaMinusZero \\
							&\Rightarrow	g(1x) \in \GammaMinusZero \\
							&\Rightarrow	0g_1(x) \in \GammaMinusZero \\
							&\Rightarrow	g_1(x) \in \GammaMinusZero^c \\
							&\Rightarrow  	g_1 \in \stab_{\sigma\omega}(\GammaMinusZero). \\
\end{align*}
\end{enumerate}
\end{proof}

We recall the following classical lemma which gives an estimate of the size of restrictions of an element which fixes the first level of the tree.
\begin{lem} \label{L:reduction_lemma}
Let $\omega$ be a sequence in $\{0,1,2\}^\infty$ and $g$ an element of $\Gom$ which stabilizes the first level of the tree. Then,
\begin{equation*}
l_{\sigma\omega}(g_i) \leq \frac{l_{\omega}(g) +1}{2} \quad i =0,1
\end{equation*}
where $l_{\omega}(g)$ (resp. $l_{\sigma\omega}$) is the length function with respect to the corresponding generating set of $\Gom$ (resp. of $\G_{\sigma\omega}$).
\end{lem}
\begin{proof}
It is a straightforward generalization of \cite[Lemma VIII.46]{DelaHarpe2000}.
\end{proof}

We can now compute explicitly the stabilizers of $\Gamma_+$ and $\GammaMinusZero$.
\begin{prop} \label{P:stab_gamma}
Let $\omega$ be a sequence without repetition. Then,  
\begin{align*}
\stabOm(\Gamma_+) &= \gen{a,u_{\omega,1}} \iso D_8\\
\stabOm(\GammaMinusZero) &= \gen{\bOm,\cOm,\dOm} \iso \Z/2\Z \times \Z/2\Z
\end{align*}
where $u_{\omega,1}$ is the generator in $\{\bOm,\cOm,\dOm\}$ such that $u_{\omega,1} 0^\infty = 0^\infty$, or explicitly,
\begin{equation*}
u_{\omega,1} = 
\begin{cases}
\bOm & \omega_1 = 2 \\
\cOm & \omega_1 = 1 \\
\dOm & \omega_1 = 0
\end{cases}.
\end{equation*}
\end{prop} 
\begin{proof}
We will proceed by induction on the length of the elements of the group  in parallel for the two parts of the proposition and all the sequences $\omega$ without repetition. 

Let us begin by some notations. We denote by $u_{\omega,n}$ the generator in $\{\bOm,\cOm,\dOm\}$ such that $u_{\omega,n}1^n0^\infty = 1^n0^\infty $, or explicitly,
\begin{equation*}
u_{\omega,n} = 
\begin{cases}
\bOm & \omega_n = 2 \\
\cOm & \omega_n = 1 \\
\dOm & \omega_n = 0
\end{cases}.
\end{equation*}
For the initial step, we verify the statements by hand for the elements of length smaller or equal to 2. It is clear that $\gen{a,u_{\omega,1}} \subset \stabOm(\Gamma_+)$ and $\gen{\bOm,\cOm,\dOm} \subset \stabOm(\GammaMinusZero)$. For all the other elements, an example of a vertex that comes out of the sets is given in Figure~\ref{F:induction_stab_Gamma}.
\begin{figure}
\centering
\begin{subfigure}{1\textwidth}
\centering
\begin{tikzpicture}
\node[sommet] (11) at (-4,0) {};
\node[nom]  at (-4,-0.75) {$110^\infty$};
\node[sommet] (01) at (-2,0) {};
\node[nom]  at (-2,-0.75) {$010^\infty$};
\node[sommet] (0) at (0,0) {};
\node[nom]  at (0,-0.75) {$0^\infty$};
\node[sommet] (10) at (2,0) {};
\node[nom]  at (2,-0.75) {$10^\infty$};
\node[sommet] (101) at (4,0) {};
\node[nom]  at (4,-0.75) {$1010^\infty$};
\node[style=empty] (L) at (-6,0) {};
\node[style=empty](R) at (6,0) {};
\draw (101) to (R) ;
\draw (10) to [bend right]  node[label] {$u$} (101);
\draw (10) to [bend left]  node[label] {$u_{\omega,1}$} (101);
\draw (10) to [loop_above]  node[label_above] {$u_{\omega,2}$} (10);
\draw (101) to [loop_above]  node[label_above] {$u_{\omega,2}$} (101);
\draw (0) to  node[label] {$a$} (10);
\draw (0) to [loop_above] node[label_above] {$u_{\omega,1}$} (0);
\draw (0) to [bend right] node[label] {$u_{\omega,2}$}  (01) ;
\draw (0) to [bend left] node[label] {$u$} (01) ;
\draw (01) to [loop_above] node[label_above] {$u_{\omega,1}$} (01) ;
\draw (01) to  node[label] {$a$} (11);
\draw (11) to [loop_above] (11) ;
\draw (11) to [bend right]  (L) ;
\draw (11) to [bend left] (L) ;
\end{tikzpicture}
\caption{The neighbourhood of $0^\infty$ on the graph $\Gamma_\omega$. The third generator which is not $u_{\omega,1} or u_{\omega,2}$ is denoted by $u$.}
\end{subfigure}
\par\bigskip
\begin{subfigure}[b]{0.4\textwidth}
\centering
\begin{tabular}{c | c}
$u_{\omega,2}$ 	& $0^\infty$ \\ \hline
$w$ 	& $0^\infty$ \\ \hline
$au_{\omega,2}$ 	& $0^\infty$ \\ \hline
$aw$ 	& $0^\infty$ \\ \hline
$u_{\omega,2}a$ 	& $10^\infty$ \\ \hline
$wa$ 	& $10^\infty$ \\
\end{tabular}
\caption{For every element of $\Gom$ of length smaller to $2$ which is not in $\gen{a,u_{\omega,1}}$, we exhibit a vertex of $\Gamma_+$ whose image is in $\Gamma_+^c$.}
\end{subfigure}
\hfill
\begin{subfigure}[b]{0.4\textwidth}
\centering
\begin{tabular}{c | c}
$a$ 	& $10^\infty$ \\ \hline
$au_{\omega,1}$ 	& $1010^\infty$ \\ \hline
$au_{\omega,2}$ 	& $10^\infty$ \\ \hline
$aw$ 	& $1010^\infty$ \\ \hline
$u_{\omega,1}a$ 	& $10^\infty$ \\ \hline
$u_{\omega,2}a$ 	& $10^\infty$ \\ \hline
$wa$ 	& $10^\infty$ \\
\end{tabular}
\caption{For every element of $\Gom$ of length smaller to $2$ of $\Gom$ which is not in $\gen{u_{\omega,1},u_{\omega,2},u}$, we exhibit a vertex of $\GammaMinusZero$ whose image is in $\GammaMinusZero^c$.}
\end{subfigure}
\caption{Initial step for the elements of small length.} \label{F:induction_stab_Gamma}
\end{figure}

We now suppose that, for a fixed integer $n\geq3$ and all the admissible sequences $\omega$, all the elements of length at most $n$ are in the desired subgroups. Recall that the stabilizer of the first level of the tree has an explicit generating set:
\begin{equation*}
 \stabOm(1) = \gen{\bOm,\cOm,\dOm,a\bOm a,a\cOm a,a \dOm a}.
\end{equation*}
The restrictions of these generators are 
\begin{align*}
\bOm &= (I_\omega^2,b_{\sigma\omega}) \\
\cOm &= (I_\omega^1,c_{\sigma\omega}) \\
\dOm &= (I_\omega^0,d_{\sigma\omega}) \\
a\bOm a &= (b_{\sigma\omega},I_\omega^2) \\
a \cOm a &= (c_{\sigma\omega},I_\omega^1) \\
a \dOm a &= (d_{\sigma\omega},I_\omega^0).
\end{align*}

Let $g$ be an element of $\Gom$ of length $n+1$. We note that if $\omega$ is a sequence without repetition, $\sigma\omega$ is also. There are three cases to prove.
\begin{enumerate}
\item{\textbf{If $g$ is in $\stabOm(\Gamma_+)$.}} 
We notice that $g$ is in $\stabOm(\Gamma_+)$ if and only if $ag$ also is. By supposing that $g$ can be of length $n+2$, we can assume that $g$ stabilizes the first level of the tree. By Lemmas~\ref{L:stab_projections} and \ref{L:reduction_lemma}, $g_0$ and $g_1$ are of length at most $n$ and are in $\stab_{\sigma \omega}(\GammaMinusZero)$. Using induction's hypothesis, $g_0$ and $g_1$ are in $\gen{b_{\sigma \omega},c_{\sigma \omega},d_{\sigma \omega}}$. As the restriction $g_0$ does not contain the generator $a$ in its writing, the element $g$ can only contain the generator $u_{\omega,1}$ of the set $\{\bOm, \cOm, \dOm\}$. In the same way, the restriction $g_1$ does not contain the generator $a$ and then the element $g$ can only contain the generator $au_{\omega,1}a$ of the set $\{a \bOm a, a\cOm a, a\dOm a\}$. Therefore, $g$ is in $\gen{u_{\omega,1},au_{\omega,1}a}$, but we may have multiply $g$ by $a$, then $g$ is in $\gen{a, u_{\omega,1}}$.

\item{\textbf{If $g$ is in $\stabOm(\GammaMinusZero)$ and stabilizes the first level.}}
Combining Lemmas~\ref{L:stab_projections} and \ref{L:reduction_lemma}, the restriction $g_0$ is in $\gen{a,u_{\sigma \omega,1}}$ and the restriction $g_1$ is in $\gen{u_{\sigma \omega,1}}$. The form of $g_0$ allows the element $g$ to contain only the generator $au_{\omega,2} a$ of the set $\{a b_{\sigma \omega} a,a c_{\sigma \omega}a,a d_{\sigma \omega}a \}$ because this is the only one which projects on $u_{\sigma \omega,1}$. The restriction $g_1$ does not contain the generator $a$, then the element $g$ can only contain the element $a u_{\omega,1} a$ of $\{a b_{\sigma \omega} a,a c_{\sigma \omega}a,a d_{\sigma \omega}a \}$. As $\omega_1 \neq \omega_2$, the element $a u_{\omega,1} a$ is not equal to the element $au_{\omega,2} a$ and then $g$ is in $\gen{\bOm, \cOm , \dOm }$.

\item{\textbf{If $g$ is in $\stabOm(\GammaMinusZero)$ and does not stabilize the first level.}}
Once again by combining Lemmas~\ref{L:stab_projections} and \ref{L:reduction_lemma}, the restrictions $g_0$ and $g_1$ are in $\gen{u_{\sigma \omega,1}}$. The form of $g_0$ allows only the generators $\{u_{\omega,1}, a u_{\omega,2} \}$ and $g_1$ only the generators $\{u_{\omega,2}, a u_{\omega,1} \}$. As $\omega_1 \neq \omega_2$, these generators are different and then $g$ is equal to $1$.
\end{enumerate}

For the explicit isomorphisms, it is left to the reader to generalize the case of $\G_{(012)^\infty}$ proved in \cite[VIII.B.10]{DelaHarpe2000} and \cite[VIII.B.16]{DelaHarpe2000}.
\end{proof}
The second step of the proof is to show that it is also difficult to stabilize subsets which differ only on a finite number of vertices with $\Gamma_+$.
\begin{prop}\label{P:Finite_symm_finite_stab}
Let $\omega$ be a sequence without repetition and $\Omega$ a subset of vertices of $\Gamma$ such that $\Omega \triangle \Gamma_+$ is finite. Then $\stab_\omega(\Omega)$ is finite.
\end{prop}
\begin{proof}
We denote by $\Lambda$ the symmetric difference $\Omega \triangle \Gamma_+$. This is a finite set, therefore there exists an even integer $n$ such that, for every $x$ in $\Lambda$, there exists $x' \in \{0,1\}^n$ such that $x=x'0^\infty$.

Let $x$ be an arbitrary prefix of length $n$ such that $x0^\infty$ is an element of $\Omega$ and let $y$ be an element of $\{0,1\}^\infty$. We claim that
\begin{equation*}
xy \in \Omega \text{ if and only if } y \in \Gamma_+.
\end{equation*}

The claim is trivial if $y = 0^\infty$. Now, let us assume that $y$ is not equal to $0^\infty$. As the length of $x$ is even, the parities of the positions of the last 1 of $y$ and of $xy$ are equal. Then $y$ is in $\Gamma_+$ if and only if $xy$ is also in $\Gamma_+$. Moreover, the element $xy$ is in $\Omega$ if and only if $xy$ is in $\Gamma_+$. Indeed, all the elements of the symmetric difference $\Lambda$ have a prefix of length shorter or equal than $n$ which is not the case of $xy$ as $y$ is not equal to $0^\infty$. The claim is therefore true.

In the same way, let $x$ be a prefix of length $n$ such that $x0^\infty$ is not an element of $\Omega$ and let $y$ be an element of $\{0,1\}^\infty$, we claim that  
\begin{equation*}
xy \in \Omega \text{ if and only if } y \in \GammaMinusZero.
\end{equation*}

Let $K= \stabOm(\Omega) \cap \stabOm(n)$ be the subgroup of elements which stabilizes $\Omega$ as well as the $n$\th{} level of the tree. For every vertex $x$ of the $n$\th{} level of the tree and every element $g$ of $K$, the restriction $g_x$ is an element of $\stab_{\sigma^n\omega}(\Gamma_+) \cup \stab_{\sigma^n\omega}(\GammaMinusZero)$. 

Indeed, if $x0^\infty$ is in $\Omega$, the following equivalences are true for all $y \in \Gamma_+$:
\begin{align*}
y \in \Gamma_+ 	&\Leftrightarrow	 xy \in \Omega \\
				&\Leftrightarrow	 g(xy) \in \Omega \\
				&\Leftrightarrow	 xg_x(y) \in \Omega \\
				&\Leftrightarrow	 g_x(y) \in \Gamma_+
\end{align*}
and then $g_x$ is in $\stab_{\sigma^n\omega}(\Gamma_+)$. In the same way, if $x0^\infty$ is not in $\Omega$: 
\begin{align*}
y \in \GammaMinusZero 	&\Leftrightarrow	 xy \in \Omega \\
										&\Leftrightarrow	 g(xy) \in \Omega \\
										&\Leftrightarrow	 xg_x(y) \in \Omega \\
										&\Leftrightarrow	 g_x(y) \in \GammaMinusZero
\end{align*}
and then $g_x$ is in $\stab_{\sigma^n\omega}(\GammaMinusZero)$.
Therefore, we can define the following embedding:
\begin{equation*}
K \hookrightarrow \prod_{v \in \{0,1\}^n} \left( \stab_{\sigma^n\omega}(\Gamma_+) \cup \stab_{\sigma^n\omega}(\GammaMinusZero) \right).
\end{equation*}
We proved in Proposition~\ref{P:stab_gamma} that $\stab_{\sigma^n\omega}(\Gamma_+)$ and $\stab_{\sigma^n\omega}(\GammaMinusZero)$ are finite and then $K$ is also finite. As $K$ is a subgroup of $\stabOm(\Omega)$ of index at most $2^n$, the stabilizer is finite. 
\end{proof}
\begin{cor}\label{C:size_stab_vertices}
Let $\omega$ be a sequence without repetition, $v$ a vertex of $\X$ with a support $\Lambda$ and $n$ an even integer such that the prefixes of the elements of $\Lambda \triangle \Gamma_+$ are at most of length $n$. Then,
\begin{equation*}
|\stabOm(v) | \leq 8 \cdot 4 \cdot 4^n 
\end{equation*}
\end{cor}
\begin{proof}
It is a direct computation using the embedding in the proof of Proposition~\ref{P:Finite_symm_finite_stab} , the cardinals of $D_8$ and $\Z/2\Z \times \Z / 2\Z$ and $[\Gom,\stabOm(n)]=2^n$.
\end{proof}

We can now prove the main theorem of this section.
\begin{mainthm}\label{T:Properness_action}
Let $\omega$ be an element of $\{0,1,2\}^\infty$ without repetition. Then the action of $\Gom$ on $\X$ is proper and faithful.
\end{mainthm}
\begin{proof}
The properness is a consequence of Proposition~\ref{P:Finite_symm_finite_stab} and the second part of Theorem~\ref{T:main_A}.

To prove that the action is faithful, we need to show that 
\begin{equation*}
\smashoperator[r]{\bigcap_{\substack{\Omega \subset V(\Gamma) \\ |\Omega \triangle V(\Gamma)|<\infty}}} \stabOm(\Omega) = \bigcap_{v \in \V} \stabOm(v)  = \{1\}. 
\end{equation*}
By Proposition~\ref{P:stab_gamma}, the intersection $\stabOm(\Gamma) \cap \stabOm(\GammaMinusZero) = \{1, \dOm\}$. Moreover, $\dOm$ does not stabilize $\Gamma_+ \setminus\{0^\infty, 1010^\infty\}$ because $\dOm 101^\infty = 10^\infty$. Then, the intersection above is trivial.
\end{proof}
\begin{rem}
The faithfulness is an automatic consequence of the properness of the action for the groups $\Gom$ which are just infinite, like the first Grigorchuk group $\G_{(012)^\infty}$ or, more generally, all the $\Gom$ with $\omega$ containing an infinite number of 0, 1 and 2, see \cite{Bartholdi2003}.
\end{rem}
We can further refine the description of the stabilizers.
\begin{prop}\label{P:subgroup_stab}
Let $\omega$ be a sequence without repetition and $H$ be a finite subgroup of $\G_\omega$. Then, there exists a vertex $v$ of $\X$ such that $H < \stabOm(v)$.
\end{prop}
\begin{proof}
It is direct an application of Proposition~\ref{P:bounded_orbits_fix_vertex}
\end{proof}
\begin{cor}
Let $\omega$ be a sequence without repetition. The stabilizers of the vertices of $\X$ for the action of $\Gom$ are arbitrarily large.
\end{cor}
\begin{proof}
The order of the elements of $\Gom$ is not bounded, then there exist finite arbitrarily large finite subgroups in $\Gom$. By the previous proposition, we can also find arbitrarily large stabilizers of vertices.
\end{proof}
It is possible to describe the distortion of infinite order elements.
\begin{cor}
Let $\omega$ be a sequence without repetition. Then, all infinite order elements of the group $\Gom$ are undistorted.
\end{cor}
\begin{proof}
It is a direct application of Haglund’s axis theorem \cite[Theorem 1.6]{Haglund2023}. 
\end{proof}
\section{Classifying space of proper actions}
\begin{defn}
A classifying space of proper actions of a group $G$, denoted by $\EG$, is a topological space which admits a proper $G$-action and with the following property: if $X$ is any space with a proper $G$-action, then there exists a $G$-equivariant map $f: X \rightarrow \EG$ and any two $G$-equivariant maps  $X \rightarrow \EG$ admit a $G$-equivariant homotopy between them.
\end{defn}
As proved in Theorem~\ref{T:Properness_action}, the cube complex $\X$ is a topological space which admits a proper $\Gom$-action if $\omega$ does not contain repetition. We will show that it satisfies the property above.
\begin{thm}\label{T:model}
Let $\omega \in \{0,1,2\}^\infty$ be a sequence without repetition. Then the CAT(0) cube complex $\X$ is a model for $\EG$, the classifying space of proper actions of $\Gom$.
\end{thm}
\begin{proof}
We use the reformulation of the universality definition proved in \cite[Prop 1.8]{Baum1994}. Applied to our situation, it means that in order to prove that CAT(0) cube complex is a model for $\EGom$, it is enough to verify the following two points.
\begin{enumerate}
\item If $H$ is a finite subgroup of $\Gom$, then there exists a point $x$ in $\X$ which is fixed by $H$.
\item View $\X \times \X$ as a space endowed with the usual diagonal action $g(x_0, x_1) = (gx_0, gx_1)$. Denote by $p_0,p_1: \X \times \X \rightarrow \X$ the two projections on each component, then there exists a $\G$-equivariant homotopy between $p_0$ and $p_1$. 
\end{enumerate}
The first point is a consequence of Proposition~\ref{P:subgroup_stab}. For the second one, we define the following homotopy 
\begin{equation*}
h((x_0,x_1),t) = t x_0 + (1-t) x_1.
\end{equation*}
This is well-defined as $\X$ is convex.
\end{proof}

\section{Further directions of research}
A group action $G \actsGroup X$ of a group $G$ on a metric space $X$ is called \emph{metrically proper} if $d(x,g_nx) \rightarrow \infty$ for every infinite family $\{g_n\}$ of elements of $\G$ and every point $x$ in $X$. This property is strictly stronger than the definition of proper action used above if the space $X$ is not locally compact. For example, consider the Cayley graph of a free group generated by a countable infinite family of generators and the canonical associated action of the group on it. All the stabilizers are trivial hence the action is proper, but all the generators send the identity at distance $1$ and so it is not metrically proper. We don't know if the actions $\Gom \actsGroup \X$ are metrically proper.
\begin{question}
Are there any sequences $\omega$ for which the action of $\Gom$ on $\X$ is metrically proper?
\end{question}
Our intuition tells us that this is the case for the groups $\Gom$ where the action is proper. A positive answer will give an elegant proof of the Haagerup property for these groups by \cite{Cherix2004}. This property is known for all Grigorchuk groups using the fact that they have subexponential growth, whereas here we would have a more elementary proof.

It is not clear to us if the condition on the sequence $\omega$ appearing in Theorem~\ref{T:Properness_action} is purely technical or if it reflects a real difference in the behaviour of actions. Indeed, we had never seen a condition of this type appear before in the study of Grigorchuk groups. It would be interesting to study the actions of the groups whose sequences have repetitions and to see if these are proper. However, it seems to us that the proof presented above is unlikely to be generalized to these cases.
\begin{question}
Are the actions of the groups $\Gom$ proper if the sequence $\omega$ has a repetition?
\end{question}

On Remark~\ref{R:dimension}, we explain why the groups $\Gom$ which are 2-groups cannot act without bounded orbit on a CAT(0) cube complex of locally finite dimension. However, we do not know any obstruction of the existence of such an action for groups with elements of infinite order.

\begin{question}
Let $\Gom$ a group which contains an element of infinite order. Is it possible to construct an action of $\Gom$ on a CAT(0) cube complex of (locally) finite dimension without bounded orbit?
\end{question}
Using the Tits' alternative proved by Caprace and Sageev \cite{Caprace2011} for groups acting on finite dimensional CAT(0) cube complexes and the fact that the $\Gom$ are amenable, the previous question can be reduced to: 
\begin{question}
May the $\Gom$ virtually surject onto $\Z$?
\end{question}

There are several notions of boundaries of a CAT(0) cube complex, like the simplicial boundary \cite{Hagen2018,Hagen2013}, the Roller boundary (which  for a cube complex $X$ coincides with $\overline{X} \setminus X$) \cite{Roller1998} and the Poisson-Furstenberg boundary \cite{Fernos2018b,Fernos2018,Nevo2013}. An interesting question would be the study of the actions we have built on these boundaries and to understand what can be deduced from it.

\begin{question}
Is it possible to understand the action of $\Gom$ on the boundaries of $\X$?
\end{question}

Grigorchuk groups are examples of branch groups \cite{Bartholdi2003}. There are other classes of branch groups which share properties used in the construction of $\X$ and tools used in the proof (Schreier graph with more than one end, restriction of the action on the subtrees, reduction lemma, etc.). 
\begin{question}
Is it possible to do the same construction for other finitely generated branch groups with a Schreier graph with at least 2 ends to obtain an action without bounded orbit, faithful, proper on a CAT(0) cube complex?
\end{question}

\section{acknowledgements}
The author is thankful to T. Nagnibeda for the idea of this project, the fruitful discussions and the careful rereading, to R. Grigorchuk, P-H. Leeman, D. Francoeur, P. Bagnoud, A. Vallette, A. Genevois and N. Matte Bon for helpful comments and ideas. We would also like to thank the anonymous referees for their valuable suggestions, which have significantly changed the presentation of this paper.

\bibliography{proper_action_grigorchuk_groups_CCC.bib}   % name your BibTeX data base

\begin{thebibliography}{10}

\bibitem{Agol2013}
I.~Agol.
\newblock {The virtual {H}aken conjecture}.
\newblock {\em Doc. Math.}, 18:1045--1087, 2013.

\bibitem{Bartholdi2000}
L.~Bartholdi and R.~I. Grigorchuk.
\newblock {On the spectrum of Hecke type operators related to some fractal
  groups}.
\newblock {\em Tr. Mat. Inst. Steklova}, 231(Din. Sist., Avtom. i Beskon.
  Gruppy):5--45, 2000.

\bibitem{Bartholdi2003}
L.~Bartholdi, R.~I. Grigorchuk, and Z.~{\v{S}}uni\'k.
\newblock {Branch groups}.
\newblock In {\em Handbook of algebra, {V}ol. 3}, volume~3 of {\em Handb.
  Algebr.}, pages 989--1112. Elsevier/North-Holland, Amsterdam, 2003.

\bibitem{Baum1994}
P.~Baum, A.~Connes, and N.~Higson.
\newblock {Classifying space for proper actions and {$K$}-theory of group
  {$C^\ast$}-algebras}.
\newblock In {\em {$C^\ast$}-algebras: 1943--1993 ({S}an {A}ntonio, {TX},
  1993)}, volume 167 of {\em Contemp. Math.}, pages 240--291. Amer. Math. Soc.,
  Providence, RI, 1994.

\bibitem{Bekka1995}
M.~B. Bekka, P.-A. Cherix, and A.~Valette.
\newblock {Proper affine isometric actions of amenable groups}.
\newblock In {\em Novikov conjectures, index theorems and rigidity, {V}ol. 2
  ({O}berwolfach, 1993)}, volume 227 of {\em London Math. Soc. Lecture Note
  Ser.}, pages 1--4. Cambridge Univ. Press, Cambridge, 1995.

\bibitem{Caprace2011}
P.-E. Caprace and M.~Sageev.
\newblock Rank rigidity for {CAT}(0) cube complexes.
\newblock {\em Geom. Funct. Anal.}, 21(4):851--891, 2011.

\bibitem{Chatterji2010}
I.~Chatterji, C.~Dru\c{t}u, and F.~Haglund.
\newblock {Kazhdan and {H}aagerup properties from the median viewpoint}.
\newblock {\em Adv. Math.}, 225(2):882--921, 2010.

\bibitem{Chatterji2005}
I.~Chatterji and G.~Niblo.
\newblock {From wall spaces to {$\rm CAT(0)$} cube complexes}.
\newblock {\em Internat. J. Algebra Comput.}, 15(5-6):875--885, 2005.

\bibitem{Cherix2004}
P.-A. Cherix, F.~Martin, and A.~Valette.
\newblock {Spaces with measured walls, the {H}aagerup property and property
  ({T})}.
\newblock {\em Ergodic Theory Dynam. Systems}, 24(6):1895--1908, 2004.

\bibitem{Cornulier2013}
Y.~Cornulier.
\newblock {Group actions with commensurated subsets, wallings and cubings}.
\newblock {\em arXiv e-prints}, page~58, 2013.

\bibitem{DelaHarpe2000}
P.~de~la Harpe.
\newblock {\em {Topics in geometric group theory}}.
\newblock Chicago Lectures in Mathematics. University of Chicago Press,
  Chicago, IL, 2000.

\bibitem{Fernos2018}
T.~Fern{\'{o}}s.
\newblock {The {F}urstenberg-{P}oisson boundary and {${\rm CAT}(0)$} cube
  complexes}.
\newblock {\em Ergodic Theory Dynam. Systems}, 38(6):2180--2223, 2018.

\bibitem{Fernos2018b}
T.~Fern{\'{o}}s, J.~L{\'{e}}cureux, and F.~Math{\'{e}}us.
\newblock {Random walks and boundaries of {$\rm CAT(0)$} cubical complexes}.
\newblock {\em Comment. Math. Helv.}, 93(2):291--333, 2018.

\bibitem{Genevois2024}
A.~Genevois, A.~Lonjou, and C.~Urech.
\newblock {Cremona Groups Over Finite Fields, Neretin Groups, and
  Non-Positively Curved Cube Complexes}.
\newblock {\em International Mathematics Research Notices}, 2024(1):554--596,
  jan 2024.

\bibitem{Grigorchuk1980}
R.~I. Grigorchuk.
\newblock {On {B}urnside's problem on periodic groups}.
\newblock {\em Funktsional. Anal. i Prilozhen.}, 14(1):53--54, 1980.

\bibitem{Grigorchuk1984}
R.~I. Grigorchuk.
\newblock {Degrees of growth of finitely generated groups and the theory of
  invariant means}.
\newblock {\em Izv. Akad. Nauk SSSR Ser. Mat.}, 48(5):939--985, 1984.

\bibitem{Grigorchuk2018}
R.~I. Grigorchuk, D.~Lenz, and T.~Nagnibeda.
\newblock {Spectra of {S}chreier graphs of {G}rigorchuk's group and
  {S}chroedinger operators with aperiodic order}.
\newblock {\em Math. Ann.}, 370(3-4):1607--1637, 2018.

\bibitem{Grigorchuk2019}
R.~I. Grigorchuk, D.~Lenz, T.~Nagnibeda, and D.~Sell.
\newblock {Subshifts with leading sequences, uniformity of cocycles and spectra
  of Schreier graphs}.
\newblock {\em arXiv e-prints}, 2019.

\bibitem{Grigorchuk2006}
R.~I. Grigorchuk and I.~Pak.
\newblock {Groups of Intermediate Growth: an Introduction for Beginners}.
\newblock {\em arXiv e-prints}, 2006.

\bibitem{Haettel2022}
T.~Haettel and D.~Osajda.
\newblock Locally elliptic actions, torsion groups, and nonpositively curved
  spaces, 2022.

\bibitem{Hagen2013}
M.~F. Hagen.
\newblock {The simplicial boundary of a {CAT}(0) cube complex}.
\newblock {\em Algebr. Geom. Topol.}, 13(3):1299--1367, 2013.

\bibitem{Hagen2018}
M.~F. Hagen.
\newblock {Corrigendum to the article {T}he simplicial boundary of a {CAT}(0)
  cube complex}.
\newblock {\em Algebr. Geom. Topol.}, 18(2):1251--1256, 2018.

\bibitem{Haglund2008}
F.~Haglund.
\newblock {Finite index subgroups of graph products}.
\newblock {\em Geom. Dedicata}, 135:167--209, 2008.

\bibitem{Haglund2023}
F.~Haglund.
\newblock Isometries of {$CAT(0)$} cube complexes are semi-simple.
\newblock {\em Ann. Math. Qu\'e.}, 47(2):249--261, 2023.

\bibitem{Haglund2012}
F.~Haglund and D.~T. Wise.
\newblock {A combination theorem for special cube complexes}.
\newblock {\em Ann. of Math. (2)}, 176(3):1427--1482, 2012.

\bibitem{Higson1997}
N.~Higson and G.~Kasparov.
\newblock {Operator {$K$}-theory for groups which act properly and
  isometrically on {H}ilbert space}.
\newblock {\em Electron. Res. Announc. Amer. Math. Soc.}, 3:131--142, 1997.

\bibitem{Leary2013}
I.~J. Leary.
\newblock {A metric {K}an-{T}hurston theorem}.
\newblock {\em J. Topol.}, 6(1):251--284, 2013.

\bibitem{MatteBon2015}
N.~{Matte Bon}.
\newblock {Topological full groups of minimal subshifts with subgroups of
  intermediate growth}.
\newblock {\em J. Mod. Dyn.}, 9:67--80, 2015.

\bibitem{Nevo2013}
A.~Nevo and M.~Sageev.
\newblock {The {P}oisson boundary of {${\rm CAT}(0)$} cube complex groups}.
\newblock {\em Groups Geom. Dyn.}, 7(3):653--695, 2013.

\bibitem{Niblo1997}
G.~Niblo and L.~Reeves.
\newblock {Groups acting on {${\rm CAT}(0)$} cube complexes}.
\newblock {\em Geom. Topol.}, 1:approx. 7 pp.\, 1997.

\bibitem{Nica2004}
B.~Nica.
\newblock {Cubulating spaces with walls}.
\newblock {\em Algebr. Geom. Topol.}, 4:297--309, 2004.

\bibitem{Roller1998}
M.~A. Roller.
\newblock {\em {Poc Sets, Median Algebras and Group Actions}}.
\newblock PhD thesis, Universit{\"{a}}t Regensburg, 1998.

\bibitem{Sageev1995}
M.~Sageev.
\newblock {Ends of group pairs and non-positively curved cube complexes}.
\newblock {\em Proc. London Math. Soc. (3)}, 71(3):585--617, 1995.

\bibitem{Sageev2005}
M.~Sageev and D.~T. Wise.
\newblock {The {T}its alternative for {${\rm CAT}(0)$} cubical complexes}.
\newblock {\em Bull. London Math. Soc.}, 37(5):706--710, 2005.

\bibitem{Valette2002}
A.~Valette.
\newblock {\em {Introduction to the {B}aum-{C}onnes conjecture}}.
\newblock Lectures in Mathematics ETH Z{\"{u}}rich. Birkh{\"{a}}user Verlag,
  Basel, 2002.

\bibitem{Vorobets2012}
Y.~Vorobets.
\newblock {Notes on the {S}chreier graphs of the {G}rigorchuk group}.
\newblock In {\em Dynamical systems and group actions}, volume 567 of {\em
  Contemp. Math.}, pages 221--248. Amer. Math. Soc., Providence, RI, 2012.

\bibitem{Wise2012}
D.~T. Wise.
\newblock {\em {From riches to raags: 3-manifolds, right-angled {A}rtin groups,
  and cubical geometry}}, volume 117 of {\em CBMS Regional Conference Series in
  Mathematics}.
\newblock Published for the Conference Board of the Mathematical Sciences,
  Washington, DC; by the American Mathematical Society, Providence, RI, 2012.

\end{thebibliography}
\bibliographystyle{abbrv}

\end{document}